\documentclass[a4paper,12pt]{article}
\usepackage[utf8]{inputenc}
\usepackage{amsmath}
\usepackage{amssymb}
\usepackage{amsthm}
\usepackage{xcolor}

\usepackage{enumerate}
\usepackage{amssymb, amsmath,graphicx}
\usepackage{todonotes}
\usepackage{verbatim}
\usepackage{float}

\theoremstyle{plain}
\newtheorem{thm}{Theorem}[section]
\newtheorem{theorem}[thm]{Theorem}

\newtheorem{lemma}[thm]{Lemma}
\newtheorem{corollary}[thm]{Corollary}

\newtheorem{remark}[thm]{Remark}

\newtheorem*{example*}{Example}
\newtheorem*{remark*}{Remark}

\def\R{{\mathbb R}}

\usepackage{amsmath, verbatim, amsfonts, amssymb, color}
\usepackage{mathrsfs}
\usepackage{dsfont}
 \usepackage{mathbbol}

\def\R{{\mathbb R}}

\begin{document}

\title{\bfseries  Gradient and Transport Estimates for Heat Flow on Nonconvex Domains}

\author{Karl-Theodor Sturm
\\[1cm]
\small Hausdorff Center for Mathematics \& Institute for Applied Mathematics\\
\small University of Bonn, Germany 
}

\maketitle

\abstract
For the Neumann heat flow on nonconvex Riemannian domains $D\subset M$, we provide sharp gradient estimates and transport estimates with a novel $\sqrt t$-dependence, 
for instance,
$$\text{Lip}( P^D_tf)\le e^{2S \, \sqrt{t/\pi}+\mathcal{O}(t)}\cdot \text{Lip} (f),$$
and we provide an equivalent characterization of the lower bound $S$ on the second fundamental form of the boundary in terms of these quantitative estimates.

\section{Introduction and Statement of Main Result}

Throughout this paper we assume:
\begin{itemize}
\item
$M$ is a complete  $n$-dimensional Riemannian manifold and $D\subset M$ is an open set with compact smooth boundary, ${\sf n}:={\sf n}_{\partial D}$ is its inward normal vector.
\item
$(P^D_t)_{t>0}$ is the heat semigroup on $\overline D$ with generator $\Delta$ and Neumann boundary conditions, $p_t(x,y)$ the Neumann heat kernel.
\item $W_q(\, .\,, \, .\,)$ denotes the $q$-Kantorovich-Wasserstein distance 
\item
$s: \partial D\to\R$ denotes the largest lower bound for the second fundamental form ${I\!\!I}_{\partial D}(X,Y):=-\langle\nabla_X\, {\sf n},Y\rangle$ on $\partial D$ and 
$k: M\to\R$ denotes the largest lower bound for the Ricci curvature on $M$ in the sense that
$${I\!\!I}_{\partial D}(v,v)\ge s(x)\, |v|^2, \quad \text{Ric}(w,w)\ge k(x)\, |w|^2\qquad \forall v,w\in T_xM, v\perp{\sf n}_{\partial D}.$$
We assume that $k$ is bounded from below on $\bar D$.
\end{itemize}


\begin{theorem}\label{main}
For any $p\in (1,\infty), q\in [1,\infty)$ and $S\in\R_+$ the following are equivalent:
\begin{itemize}
\item[(i)] $s\ge -S$
\item[(ii)] $\forall  t\in\R_+, \, \mu,\nu\in\mathcal{P}(\bar D)$,
$$W_q(P^D_t\mu,P^D_t\nu)\le e^{2S \, \sqrt{t/\pi}+\mathcal{O}(t)}\cdot W_q(\mu,\nu)$$
\item[(iii)] $\forall  t\in\R_+, \, f\in\mathcal{C}^1(\bar D), \, x\in \bar D$,
$$\big|\nabla P^D_tf|(x)\le e^{2S \, \sqrt{t/\pi}+\mathcal{O}(t)}\cdot P^D_t\big(|\nabla f|^p\big)^{1/p}(x)$$
\item[(iv)]
$\forall  t\in\R_+, \, f\in\mathcal{C}^1(\bar D)$,
$$\text{Lip}( P^D_tf)\le e^{2S \, \sqrt{t/\pi}+\mathcal{O}(t)}\cdot \text{Lip} (f).$$
\end{itemize}
\end{theorem}

\begin{remark}
{\bf a)}
The precise meaning of $\mathcal{O}(t)$ in (ii) 
 is: 
\begin{itemize}
\item[(ii')] $\exists C, t_0>0:  \forall  t\in [0,t_0], \, \mu,\nu\in\mathcal{P}(\bar D)$,
$$W_q(P^D_t\mu,P^D_t\nu)\le e^{2S \, \sqrt{t/\pi}+Ct}\cdot W_q(\mu,\nu)$$
\end{itemize}
and similarly for (iii) and (iv).

{\bf b)}
In (iii) and (iii'), the class $\mathcal{C}^1(\bar D)$ can equivalently be replaced by $\mathcal{C}^\infty(\bar D)$ or by $W^{1,p}(\bar D)$.
In (iv) and (iv'),  $\mathcal{C}^1(\bar D)$ can equivalently be replaced by $\mathcal{C}^\infty(\bar D)$ or by $\text{Lip}(\bar D)$.

{\bf c)}
An estimate of the type
$$\text{Lip}( P^D_tf)\le e^{C_0 \sqrt{t}+C_1t}\cdot \text{Lip} (f).$$
appeared first in \cite{St-distr} for the Neumann heat semigroup on the complement of the ball in $\R^n$. The constant $C_0$ there was not calculated explicitly.
\end{remark}

\begin{remark} This characterization of lower bounds for the second fundamental form in terms of the Neumann heat semigroup is very much in the spirit of the {\bf von Renesse-Sturm Theorem \cite{vRS05}}
on the characterization of lower Ricci bounds in terms of the heat semigroup $(P^M_t)_{t>0}$ on $M$:

For any $p\in [1,\infty), q\in [1,\infty]$ and $K\in\R$ the following are equivalent:
\begin{itemize}
\item[(i)] $k\ge -K$
\item[(ii)] $\forall  t\in\R_+, \, \mu,\nu\in\mathcal{P}(M)$,
$$W_q(P^M_t\mu,P^M_t\nu)\le e^{Kt}\cdot W_q(\mu,\nu)$$
\item[(iii)] $\forall  t\in\R_+, \, f\in\mathcal{C}^1(M), \, x\in M$,
$$\big|\nabla P^M_tf|(x)\le e^{Kt}\cdot P^M_t\big(|\nabla f|^p\big)^{1/p}(x)$$
\item[(iv)]
$\forall  t\in\R_+, \, f\in\mathcal{C}^1(M)$,
$$\text{Lip}( P^M_tf)\le e^{Kt}\cdot \text{Lip} (f).$$
\end{itemize}
\end{remark}

\begin{remark} Analogous results hold true for the Neumann heat semigroup  $(P^D_t)_{t>0}$ for any domain $D\subset M$ with smooth boundary provided $D$ is convex --- whereas none of the previous assertions (ii), (iii) and (iv) holds true if $D$ is non-convex.
More precisely, for any $p\in [1,\infty), q\in [1,\infty]$ and $K\in\R$ the following are equivalent:
\begin{itemize}
\item[(i)] $k\ge -K$ on $\bar D$ and $D$ is convex
\item[(ii)] $\forall  t\in\R_+, \, \mu,\nu\in\mathcal{P}(\bar D)$,
$$W_q(P^D_t\mu,P^D_t\nu)\le e^{Kt}\cdot W_q(\mu,\nu)$$
\item[(iii)] $\forall  t\in\R_+, \, f\in\mathcal{C}^1(\bar D), \, x\in \bar D$,
$$\big|\nabla P^D_tf|(x)\le e^{Kt}\cdot P^D_t\big(|\nabla f|^p\big)^{1/p}(x)$$
\item[(iv)]
$\forall  t\in\R_+, \, f\in\mathcal{C}^1(\bar D)$,
$$\text{Lip}( P^D_tf)\le e^{Kt}\cdot \text{Lip} (f).$$
\end{itemize}
\end{remark}

\bigskip

\noindent
{\bf Acknowledgement:} I am very grateful to Mathav Murugan for a kind discussion during the MeRiOT conference 2024 at Lake Como which provided important arguments and stimuli for this research.

\section{Ingredients}
In the sequel we fix $D$ and simply denote $P^D_t$ by $P_t$.
Let $(B_t,\mathbb P_x)_{t\in\R_+, x\in \bar D}$ denote the rescaled reflected Brownian motion associated with the Neumann heat semigroup $(P_t)_{t>0}$ with generator  $\Delta$, and let $(\ell_t)_{t>0}$ denote the associated local time of $\partial D$.
Recall that the latter is characterized by the fact that
\begin{equation}\label{ito}
f(B_t)=f(B_0)+ M^f_t+\int_0^t \Delta f(B_r)\,dr + \int_0^t \nabla_{\sf n} f(B_r)\,d\ell_r\end{equation}
for any $f\in \mathcal C^2(\bar D)$ and a suitable martingale $(M^f_t)_{t>0}$. In the case $M=\R^n$, the martingale is explicitly given as $M^f_t=\int_0^t \nabla f(B_r)\,dB_r$.

\begin{lemma}\label{l3}
 $(\ell_t)_{t>0}$  is the additive functional associated with the surface measure $\sigma$, that is,
\begin{equation}\label{pcaf} \mathbb E_x\left[\int_0^tf(B_r)d\ell_r\right]=\int_{\partial D}\int_0^t f(y) p_r(x,y)dr\,d\sigma(y) \qquad \forall f, t, x.
\end{equation}
\end{lemma}
Thus in particular
\begin{equation*}\mathbb E_x[\ell_t]=\int_{\partial D}\int_0^t p_r(x,y)dr\,d\sigma(y).\end{equation*}

\begin{proof} The result is proven in the Euclidean setting in \cite{BH90}. The extension to the Riemannian setting is straightforward, see also  \cite{Ch93}.
\end{proof}
\begin{proof}[Heuristic argumentation] 
Let $\psi_\epsilon(x)=\Big(1\wedge \frac1\epsilon\,d(x,M\setminus D)\Big)$ and $\rho_\epsilon(x)=\frac1\epsilon\, {\bf 1}_{(0,\epsilon]}(d(x,M\setminus D))$. Then
$$\rho_\epsilon(dx)dx\stackrel{\epsilon\to0}\longrightarrow \sigma(dx).$$
Thus the positive continuous additive functional (PCAF) associated with the process $(B_t,\mathbb P_x)_{t\in\R_+, x\in \bar D}$ and the measure $\sigma$ is given by
$$A_t=\lim_{\epsilon\to 0}\int_0^t \rho_\epsilon(B_r)dr.$$

 On the level of Dirichlet forms, the reflected Brownian motion $(B_t)_{t>0}$ is the limit of the diffusions with drift $(B^\epsilon_t)_{t>0}$:
$$\mathcal E_\epsilon\big(f,f\big):=\frac12\int_{D} |\nabla f|^2 \psi_\epsilon dx\quad\text{on }L^2(D,   dx)$$
for $\epsilon\to0$ converges in the sense of Mosco to
$$\mathcal E\big(f,f\big):=\frac12\int_{D} |\nabla f|^2  dx\quad\text{on }L^2( D,  dx).$$
The diffusion with drift satisfies the SDE
$$f(B_t^\epsilon)=f(B^\epsilon_0)+\text{martingale} + \int_0^t\Big[(\psi_\epsilon\Delta f)(B^\epsilon_r)+(\nabla\psi_\epsilon\cdot \nabla f)(B^\epsilon_r)\Big]dr.$$
Extend the definition of ${\sf n}:\partial D\to TM$ by setting ${\sf n}(x):=\nabla d(x,M\setminus D)$ for $x\in D$ and 
observe that
 $$\nabla\psi_\epsilon\cdot \nabla f\stackrel{\epsilon\to0}\longrightarrow \rho_\epsilon \cdot \nabla_{\sf n}f.$$
 Taking the convergence $B^\epsilon_t\to B_t$ for granted  and recalling that $\rho_\epsilon(B_r)ds \to dA_r$,
we obtain
 $$f(B_t)=f(B_0)+\text{martingale} + \int_0^t\Delta f(B_r)dr+\int_0^t  \nabla_{\sf n}f(B_r)dA_r.$$
 Thus, by the very definition of the local time,  $\ell_t=A_t$.
 \end{proof}

\begin{lemma}[{\cite{Wang-Book}, Thm.~3.3.1}]  $\forall f, x,t$:
\begin{align}\label{wang}
\big| \nabla P_{t}f\big|(x)\le
\mathbb E_{x}\Big[ e^{-\int_0^{t}  k(B_{r})dr-\int_0^{t} s(B_r)d\ell_r}
\cdot \big|\nabla f(B_{t})\big|\Big].
\end{align}
\end{lemma}
For an analytic proof, even in the general setting of metric measure spaces, see \cite{St-distr}, Corollary 6.15.

\begin{lemma}[{\cite{Wang-Book}, Lemma 3.1.2}]\label{l2}
\begin{equation}
\sup_{x\in \bar D} \mathbb E_x\big[ \ell_t\big]=2 \sqrt{t/\pi} + {\mathcal O}(t^{3/2}).
\end{equation}
\end{lemma}

\begin{proof}[Sketch of proof]
\begin{align*}
\sup_{x\in \bar D} \mathbb E_x\big[ \ell_t\big]&= \sup_{x\in \bar D}\int_{\partial D} \int_0^t p_r(x,y)dr\,\sigma(dy)
=
 \sup_{z\in \partial D}\int_{\partial D} \int_0^t p_r(z,y)dr\,\sigma(dy)\\&
=  \int_0^t p^{\R_+}_r(0,0)dr+ {\mathcal O}(t^{3/2})=\int_0^t\frac2{\sqrt{4\pi r}}dr+ {\mathcal O}(t^{3/2})=
2 \sqrt{t/\pi} + {\mathcal O}(t^{3/2}).
\end{align*}

\end{proof}

\begin{remark}
Let  $(\tilde P_t)_{t>0}$ denote  the Neumann heat semigroup with generator  $\frac12\Delta$, let $(\tilde B_t,\mathbb P_x)_{t\in\R_+, x\in \bar D}$ denote the associated (``standard'') reflected Brownian motion, and let $(\tilde \ell_t)_{t>0}$ denote the associated local time of $\partial D$. Then 
$$\tilde P_tf(x)=P_{t/2}f(x), \quad \tilde B_t=B_{t/2},\quad \tilde\ell_t=\ell_{t/2}.$$
Hence, the estimate from the previous Lemma \ref{wang} can be rewritten as
\begin{align}\label{hsu}
\big| \nabla \tilde P_{t}f\big|(x)\le
\mathbb E_{x}\Big[ e^{-\frac12\int_0^{t}  k(\tilde B_{r})dr-\int_0^{t} s(\tilde B_r)d\tilde\ell_r}
\cdot \big|\nabla f(\tilde B_{t})\big|\Big],
\end{align}
in accordance with \cite{Hs02}, Thm.5.1, and Lemma \ref{l2} reads
\begin{equation}
\sup_{x\in \bar D} \mathbb E_x\big[ \tilde\ell_t\big]=\sqrt{2t/\pi} + {\mathcal O}(t).
\end{equation}
Note that with  these conventions, $(\ell_t)_{t>0}$  is the additive functional associated with the process $(B_t,\mathbb P_x)_{t\in\R_+, x\in \bar D}$ and the measure $\sigma$, whereas 
 $(\tilde\ell_t)_{t>0}$  is $\frac12$ times the additive functional associated with $(\tilde B_t,\mathbb P_x)_{t\in\R_+, x\in \bar D}$ and  $\sigma$. In particular,
$$\ell_t=\lim_{\epsilon\to 0}\int_0^t \rho_\epsilon(B_s)ds, \qquad \tilde\ell_t=\frac12\cdot\lim_{\epsilon\to 0}\int_0^t \rho_\epsilon(\tilde B_r)dr$$
 for functions $\rho_\epsilon$ as above, and
 $$\mathbb E_x\big[ \ell_t\big]=\int_0^t \int_{\bar D}p_r(x,y)d\sigma(y)\,dr, \qquad \mathbb E_x\big[ \tilde\ell_t\big]=\frac12 \int_0^t \int_{\bar D}\tilde p_r(x,y)d\sigma(y)\,dr.
 $$
 \end{remark}

\section{Proof}
The {\sf proof of Theorem \ref{main}} consists of
wo nontrivial implications, namely (i)$\Rightarrow$(iii) and (iv)$\Rightarrow$(i), 
which will be addressed in the subsections below,
and some straightforward implications to be discussed right now.

\begin{description}
\item[(iii)$\Rightarrow$(iv):] Trivial
\item[(iii)$\Leftrightarrow$(ii):] Kuwada duality \cite{Ku10}.
\end{description}

\subsection{Forward implication}

\begin{theorem} Assume $s\ge -S$ for some $S\ge0$. Then for every $p\in (1,\infty)$,
\begin{align}
\big| \nabla P_{t}f\big|(x)\le e^{2S \sqrt{t/\pi}+ \mathcal{O}(t)}\cdot   P_{t}\big(|\nabla f|^p\big)^{1/p}(x).
\end{align}
\end{theorem}

\begin{proof} 
Put $K:=-\inf_{x\in \bar D} k(x)$ and $q$ such that $1/p+1/q=1$. Then according to \eqref{wang},
\begin{align*}
\big| \nabla P_{t}f\big|(x)&\le \mathbb E_{x}\Big[ e^{-\frac{q}2\int_0^{t}  k(B_{r})dr-q\int_0^{t} s(B_r)d\ell_r}
\Big]^{1/q}\cdot \mathbb E_{x}\Big[ \big|\nabla f(B_{t})\big|^p\Big]^{1/p}\\
&\le C_t\cdot   P_{t}\big(|\nabla f|^p\big)^{1/p}(x).
\end{align*}
with 
\begin{align*} C_t:=e^{Kt/2}\cdot \sup_{x\in \bar D}\mathbb E_{x}\Big[ e^{q S \cdot\ell_t}\Big]^{1/q}.
\end{align*}

 According to Khasminskii's Lemma \cite{Kh59}, straightforward extension to additive functionals, and Lemma \ref{l3},
\begin{align} \sup_{x\in \bar D}\mathbb E_{x}\Big[ e^{q S \cdot\ell_t}\Big] \le
\frac1{1-qS \cdot \sup_{x\in \bar D}\mathbb E_{x}\big[\ell_t\big]}
\end{align}
provided $\sup_{x\in \bar D}\mathbb E_{x}\big[\ell_t\big]<\frac1{qS}$.
Combined with Lemma  \ref{l2} thus
\begin{align*} \sup_{x\in \bar D}\mathbb E_{x}\Big[ e^{q S \cdot\ell_t}\Big]^{1/q} &\le \left(
\frac1{1-qS \, \big[2\sqrt{t/\pi}+ {\mathcal O}(t)\big]}\right)^{1/q}\\
&=\exp\left(2S \, \sqrt{t/\pi}+ {\mathcal O}(t)
\right).
\end{align*}
\end{proof}

\begin{corollary}
Assume $s\ge -S$ for some $S\ge0$. Then
\begin{align}
\text{Lip}  ( P_{t}f)\le e^{2S \sqrt{t/\pi}  + \mathcal{O}(t)}\cdot   \text{Lip} (f).
\end{align}
\end{corollary}

\subsection{Backward implication}

\begin{theorem}
Assume 
\begin{align}\label{grad-ass}
\text{Lip} ( P_{t}f)\le e^{2S \sqrt{t/\pi}  + \mathcal{O}(t)}\cdot   \text{Lip} (f)
\end{align}
 for some $S\ge0$.
Then $s\ge -S$.
\end{theorem} 

\begin{proof}
Assume $s\not\ge -S$, that is, $s(z_0)\le-S_1$ for some $S_1>S$ and $z_0\in\partial D$, and thus
${I\!\!I}_{\partial D}(v,v)\le -S_1 |v|^2$ for some tangent vector $v\in T_{z_0}^{\perp}$.
Choose normal coordinates around $z_0=:0$ and assume without restriction $v=e_1$. Then 
$$D=\{x\in\R^n: x_n>\psi(x_1,\ldots, x_{n-1})\}$$
(at least when intersected with a ball around $0$)
for some smooth function $\psi$
with
$\psi(x_1,\ldots, x_{n-1})=-\frac12 \sum_{i=1}^{n-1} S_ix_i^2 + \mathcal{O}(|x|^3)$ with $S_1>S\ge0$ as above and arbitrary $S_2,\ldots, S_{n-1}\in\R$.
Thus 
$${\sf n}(z)= \frac1{\sqrt{1+\sum_{i=1}^{n-1} S_i^2x_i^2}} \Big(S_1x_1, \ldots, S_{{n-1}}x_{x_{n-1}}, 1\Big)+ \mathcal{O}(|x|^2)$$
for $z\in\partial D$.

If we choose $f(x):=x_1$ then $|\nabla f|(x)=1+\mathcal O(|x|^2).$
Thus better choose $f(x):=x_1(1-C|x|^2)$ with suitable $C>0$ in a small neighborhood of 0 and bounded smooth outside. 
Then $|\nabla f|(x)\le1$
and
$$\nabla_{\sf n}f(z)= \frac1{\sqrt{1+\sum_{i=1}^{n-1} S_i^2x_i^2}} \, S_1x_1+\mathcal{O}(|x|^2)=S_1x_1+\mathcal{O}(|x|^2).$$

Consider the curve $r\mapsto z(r):=(r,0,\ldots,0, \psi(r,0,\ldots,0))\in\partial D$.
Then
\begin{align*}
P_tf(z_r)&=f(z_r)+\mathbb E_{z_r}\left[\int_0^t \Delta f(B_s)ds+\int_0^t\nabla_{\sf n}f(B_s)d\ell_s\right]\\
&=r+\mathcal{O}(t^{3/2})+\mathbb E_{z_r}\left[\int_0^t \nabla_{\sf n}f(B_s)d\ell_s\right]
\end{align*}
since $f(z_r)=r+ \mathcal{O}(r^3)$ and $\Delta f(x)=\mathcal{O}(|x|)$, thus $\mathbb E_{z_r}\left[\Delta f(B_t)\right]=\mathcal{O}(t^{1/2})$ and
$\int_0^t \mathbb E_{z_r}\left[\Delta f(B_s)\right]ds=\mathcal{O}(t^{3/2})$.
To deal with the contribution of the local time, we decompose it into its contribution on the ball $K_r:=\{x: d(x,z_r)\le \epsilon r\}$ and on the complement of it (for fixed small $\epsilon>0$) and use the fact that
$$ \nabla_{\sf n}f\ge S_1r(1-\epsilon)+\mathcal{O}(r^3)- C_0{\bf 1}_{\bar D\setminus K_r}$$
for some constant $C_0$.
Now for given $\delta>0$ choose $r=C\sqrt t$ with $C$ sufficiently large such that  $C\ge C_0\delta$ and $\int_0^1 \int_{C}^\infty \frac4{\sqrt{4\pi s}}\exp(-\frac{x^2}{4s})dxds\le\delta$. Then similar as in Lemma \ref{l2} we conclude
\begin{align*}
\mathbb E_{z_r}\left[\int_0^{t} {\bf 1}_{\bar D\setminus K_r}(B_s)d\ell_s\right]&=
\int_0^t \int_{r}^\infty \frac4{\sqrt{4\pi s}}\exp\left(-\frac{x^2}{4s}\right)dxds+\mathcal{O}(t^3)\\
&= t\cdot 
\int_0^1 \int_{C}^\infty \frac4{\sqrt{4\pi s}}\exp\left(-\frac{x^2}{4s}\right)dxds+\mathcal{O}(t^3)\\&\le\delta\,t+\mathcal{O}(t^3).
\end{align*}
Thus
\begin{align*}
\mathbb E_{z_r}\left[\int_0^{t} \nabla_{\sf n}f(B_s)d\ell_s\right]&\ge \big(S_1r(1-\epsilon)+\mathcal{O}(r^3)\big)\cdot \mathbb E_{z_r}\big[\ell_{t}\big]
-C_0\, \mathbb E_{z_r}\left[\int_0^{t} {\bf 1}_{\bar D\setminus K_r}(B_s)d\ell_s\right]
\\
&= 2S_1 r(1-\epsilon)\sqrt{t/\pi}+ \mathcal{O}(r^3)+\mathcal{O}(t^{3/2})-C_0\delta t.
\end{align*}

Since $P_tf(0)=0+\mathcal{O}(t^{3/2})$ and $d(0,z_r)=r+ \mathcal{O}(r^3)$ this amounts to
\begin{align*}
\text{Lip}(P_tf)&\quad\ge\quad
\frac{P_tf(z_r)-P_tf(0)}{d(z_r,0)}\\&\quad \ge \quad
 \frac{r+2S_1 r(1-\epsilon)\sqrt{t/\pi}-C_0\delta t+\mathcal{O}(r^3)+\mathcal{O}(t^{3/2})}{r+ \mathcal{O}(r^3)}\\
&\stackrel{r:=C\sqrt t}{=}1+2S_1(1-\epsilon)\cdot \sqrt{{t}/{\pi }}-\delta\sqrt t+\mathcal{O}(t)\\
&\quad=\quad\exp\big(2S_0\cdot \sqrt{{t}/{\pi }}+\mathcal{O}(t)\big)
\end{align*}
with $S_0:=S_1(1-\epsilon)-\sqrt\pi \delta/2$. Choosing $\delta$ and $\epsilon$ sufficiently small, yields $S_0>S$.
This contradicts assumption \eqref{grad-ass}.
\end{proof}


\begin{thebibliography}{St05a}


 \bibitem[BH90]{BH90}
 \textsc{Richard F.~ Bass, Pei Hsu (1990)}:
 The semimartingale structure of reflecting Brownian motion.
 \textit{Proc.~AMS}, 1007-1010.
 
  \bibitem[Ch93]{Ch93}
 \textsc{Zhen-Qing Chen}:
On reflecting diffusion processes and
Skorokhod decompositions.
 \textit{Probab.~Theory Relat.~Fields} {\bf 94}, 281--315.
 
 \bibitem[Hs02]{Hs02}
 \textsc{Elton Hsu (2002}:
 Multiplicative functional for the heat equation on manifolds with boundary. 
 {\bf Michigan Math.~J.} {\bf 50.2} 351--367.
  
  \bibitem[Kh59]{Kh59}
 \textsc{R.~Z.~Khas'miniskii (1959)}: On positive solutions of the equation Uu + Vu = 0.
 \textit{Theoret.~Probability
Appl.} {\bf 4} 309--318.


 \bibitem[Ku10]{Ku10}
 \textsc{Kazumasa Kuwada (2010}:
 Duality on gradient estimates and Wasserstein controls.
 \textit{J.~Functional Analysis} {\bf 258.11} 3758--3774.

 
  \bibitem[vRS05]{vRS05}
 \textsc{M.-K. von Renesse, K.~T.~Sturm (2005)}:
{Transport inequalities, gradient estimates, entropy and Ricci curvature}.
\textit{Comm.~Pure Appl.~Math.}
  {\bf 58},
 {923--940},
 
  \bibitem[St20]{St-distr}
   \textsc{K.~T.~Sturm (2020)}:
Distribution-valued Ricci bounds for metric measure spaces, singular time changes, and gradient estimates for Neumann heat flows. \textit{Geometric and Functional Analysis} {\bf 30}, 1648--1711.

 \bibitem[Wa14]{Wang-Book}
   \textsc{Feng-Yu Wang (2014)}:
    Analysis for diffusion processes on Riemannian manifolds. 
    \textit{World Scientific.}
\end{thebibliography}
\end{document}